\documentclass[11pt,twoside,english]{amsart}
\advance\oddsidemargin by -1.0cm
\advance\evensidemargin by -1.0cm
\advance\topmargin by -1.0cm 
\usepackage{amssymb}
\usepackage{babel}
\usepackage{amstext}
\usepackage{amscd}   
\usepackage{epsfig}  
\usepackage{rotating}
\let\mathg\mathfrak
\theoremstyle{plain}
\newtheorem{cor}{Corollary}[section]
\newtheorem{lem}{Lemma}[section]
\newtheorem{thm}{Theorem}[section]            
\newtheorem{prop}{Proposition}[section]
\theoremstyle{definition}

\newtheorem{NB}{Remark}[section]
\newtheorem{dfn}{Definition}[section]

\newcommand{\bdm}{\begin{displaymath}}
\newcommand{\edm}{\end{displaymath}}
\newcommand{\be}{\begin{equation}}
\newcommand{\ee}{\end{equation}}
\newcommand{\ba}[1]{\begin{array}{#1}}
\newcommand{\ea}{\end{array}}

\newcommand{\btab}{\begin{tabular}}
\newcommand{\etab}{\end{tabular}}

\newcommand{\R}{\ensuremath{\mathbb{R}}}

\newcommand{\G}{\ensuremath{\mathrm{G}}}

\newcommand{\su}{\ensuremath{\mathg{su}}}
\newcommand{\SU}{\ensuremath{\mathrm{SU}}}

\newcommand{\so}{\ensuremath{\mathg{so}}}
\newcommand{\hol}{\ensuremath{\mathg{hol}}}
\newcommand{\SO}{\ensuremath{\mathrm{SO}}}
\newcommand{\Spin}{\ensuremath{\mathrm{Spin}}}

\newcommand{\g}{\ensuremath{\mathfrak{g}}}

\begin{document}
\def\haken{\mathbin{\hbox to 6pt{%
                 \vrule height0.4pt width5pt depth0pt
                 \kern-.4pt
                 \vrule height6pt width0.4pt depth0pt\hss}}}
    \let \hook\intprod
\setcounter{equation}{0}
%
%
\thispagestyle{empty}
%
\date{\today}
\title[$3$-Sasakian manifolds in dimension seven]
{$3$-Sasakian manifolds in dimension seven, their spinors and  $\G_2$-structures}
%
%
%
\author{Ilka Agricola and Thomas Friedrich}
\address{\hspace{-5mm} 
Ilka Agricola\newline
Fachbereich Mathematik und Informatik \newline
Philipps-Universit\"at Marburg\newline
Hans-Meerwein-Strasse / Campus Lahnberge\newline
D-35032 Marburg, Germany\newline
{\normalfont\ttfamily agricola@mathematik.uni-marburg.de}}
\address{\hspace{-5mm} 
Thomas Friedrich\newline
Institut f\"ur Mathematik \newline
Humboldt-Universit\"at zu Berlin\newline
Sitz: WBC Adlershof\newline
D-10099 Berlin, Germany\newline
{\normalfont\ttfamily friedric@mathematik.hu-berlin.de}}
%
\subjclass[2000]{Primary 53 C 25; Secondary 81 T 30}
\begin{abstract}
It is well-known that $7$-dimensional $3$-Sasakian manifolds 
carry a one-parametric family of compatible $\G_2$-structures and
that they do not 
admit a characteristic connection.  In this note, we show that there is
nevertheless a distinguished cocalibrated $\G_2$-structure in this 
family whose characteristic connection $\nabla^c$ along with its
parallel spinor field $\Psi_0$ can be used
for a thorough investigation of the geometric properties of 
$7$-dimensional $3$-Sasakian manifolds.  Many known
and some new properties  can 
be easily derived from the properties of $\nabla^c$ and of $\Psi_0$, yielding 
thus an appropriate substitute for the missing characteristic connection.
\end{abstract}
\maketitle
\pagestyle{headings}
%
%
\section{Introduction}\noindent
%
$3$-Sasakian manifolds have been studied by the Japanese school in
Differential Geometry decades ago \cite{IK}. 
They are Einstein spaces of positive
scalar curvature carrying three compatible orthogonal Sasakian structures. 
In the middle of the 80-ties, a relation between $3$-Sasakian
manifolds and the spectrum of the Dirac operator was discovered
 \cite{FrKa1}, \cite{FrKa2}. Indeed, they
admit three Riemannian Killing spinors, which realize the lower bound
for the eigenvalues of the Dirac operator \cite{Fri1}. 
Seven-dimensional, regular $3$-Sasakian manifolds are classified in
\cite{FrKa1}. In the 90-ties, many new families of non-regular $3$-Sasakian
manifolds have been constructed specially in dimension seven \cite{BG}.
This dimension is important because the exceptional Lie group $\G_2$ 
admits a $7$-dimensional representation and any $3$-Sasakian-structure on a
Riemannian manifold induces a family of adapted, non-integrable 
$\G_2$-structures. A deformation of one of these $\G_2$-structures---we call 
it the {\it canonical $\G_2$-structure}---yields examples of
$7$-dimensional Riemannian manifolds with precisely one Killing spinor
\cite{FKMS}. The whole family of underlying $\G_2$-structures has been
investigated  from the viewpoint
of spin geometry in \cite{AgFr1}, section $8$. In particular, they are
solutions of type II string theory with $4$-fluxes (see \cite{Srni} for more
background and motivation).\\

\noindent
We will show that the canonical $\G_2$-structure of a $3$-Sasakian manifold is cocalibrated.
Consequently, there exists a unique connection with totally skew-symmetric
torsion preserving it, see \cite{FriedrichIvanov},
\cite{FriedrichIvanov2}. The aim of this note is to study this
characteristic connection $\nabla^c$ as well as the corresponding
$\nabla^c$-parallel spinor field $\Psi_0$. This point of view allows us to
prove  many properties of $3$-Sasakian manifolds  in
a unified way. For example, the Riemannian Killing spinors are the Clifford
products of the canonical spinor $\Psi_0$ by the three unit vectors 
defining  
the $3$-Sasakian structure: in this sense, the $\nabla^c$-parallel spinor 
field $\Psi_0$ is more fundamental than the Killing spinors. Finally we 
study the spinorial field equations and the
deformations of the canonical $\G_2$-structure in more detail. 
%
\section{$3$-Sasakian manifolds in dimension seven}\noindent
%
A $7$-dimensional {\it Sasakian manifold} is a Riemannian manifold $(M^7 , g)$ equipped with a contact form $\eta$, its dual vector field $\xi$  as
well as with an endomorphism
$\varphi : T M^7 \rightarrow  T M^7$ such that the following
conditions are satisfied:
\begin{gather*}
\eta \wedge (d \eta)^3 \neq 0  , \quad \eta(\xi) \ = \ 1 , \quad g(\xi ,
 \xi ) \ = \ 1  , \\
g(\varphi X  , \varphi Y) = g(X , Y) - \eta(X) \cdot \eta(Y)
, \quad \varphi^2 \ = \ - \mathrm{Id} \, + \, \eta \otimes \xi  ,\\
\nabla^g_X \xi = - \, \varphi X  , \quad (\nabla^g_X \varphi)(Y) \ = \ g(X , Y) \cdot \xi -  \eta(Y) \cdot X  .
\end{gather*}
These conditions imply several further relations, for example
\bdm
\varphi \xi \ = \ 0  , \quad \eta \circ \varphi \ = \ 0 , \quad
d \eta(X  ,  Y) \ = \ 2 \cdot g(X  ,  \varphi Y)  .
\edm
A $7$-dimensional {\it $3$-Sasakian manifold} is a Riemannian manifold 
$(M^7, g)$ equipped with three Sasakian structures $(\xi_{\alpha} , 
\eta_{\alpha}, \varphi_{\alpha}), \alpha=1,2,3,$ such that
\bdm
[ \xi_1 \, , \, \xi_2] \ = \ 2 \, \xi_3  , \quad
[ \xi_2 \, , \, \xi_3] \ = \ 2 \, \xi_1  , \quad
[ \xi_3 \, , \, \xi_1] \ = \ 2 \, \xi_2  
\edm
and
\begin{eqnarray*}
\varphi_3 \circ \varphi_2 &=& - \, \varphi_1  +  \eta_2 \otimes \xi_3 
, \quad \varphi_2 \circ \varphi_3 \ = \ \varphi_1  +  \eta_3 \otimes \xi_2 , \\ 
\varphi_1 \circ \varphi_3 &=& - \, \varphi_2  +  \eta_3 \otimes \xi_1 , \quad 
\varphi_3 \circ \varphi_1 \ = \ \varphi_2  +  \eta_1 \otimes \xi_3 , \\
\varphi_2 \circ \varphi_1 &=& - \, \varphi_3  +  \eta_1 \otimes \xi_2, \quad 
\varphi_1 \circ \varphi_2 \ = \ \varphi_3  +  \eta_2 \otimes \xi_1  .
 \end{eqnarray*}
The vertical subbundle $\mathrm{T}^v \subset TM^7$ is spanned by 
$\xi_1, \xi_2, \xi_3$, its orthogonal complement is the horizontal 
subbundle $\mathrm{T}^h$. Both
subbundles are invariant under $\varphi_1, \varphi_2, \varphi_3$.\\

\noindent
The properties as well as examples of Sasakian and $3$-Sasakian manifolds are
the topic of the book \cite{BG}. $3$-Sasakian manifolds are always Einstein 
with scalar curvature $R =
42$. If they are complete, they are compact with  finite fundamental
group. Therefore we shall always assume that $M^7$ is compact and
simply-connected. 
The frame bundle has a topological reduction to the subgroup $\SU(2)
\subset \SO(7)$. In particular, $M^7$ is a spin manifold. Moreover, 
there exists locally an orthonormal frame $e_1, \ldots, e_7$ such that
$e_1 = \xi_1 , \, e_2 = \xi_2 , \, e_3 = \xi_3$ and the endomorphisms
$\varphi_{\alpha}$ acting on the horizontal part $\mathrm{T}^h :=
\mathrm{Lin}(e_4,e_5,e_6,e_7)$ of the tangent bundle are given by the
following matrices 
\bdm
\varphi_1:=\begin{bmatrix} 0 & -1 & 0 & 0 \\ 1 & 0 & 0 & 0 \\ 0 & 0 & 0 & -1\\
0 & 0 & 1 & 0\end{bmatrix},\
\varphi_2:=\begin{bmatrix} 0 & 0 & -1 & 0 \\ 0 & 0 & 0 & 1 \\ 1 & 0 & 0 & 0\\
0 & -1 & 0 & 0\end{bmatrix},\
\varphi_3:=\begin{bmatrix} 0 & 0 & 0 & -1 \\ 0 & 0 & -1 & 0 \\ 0 & 1 & 0 & 0\\
1 & 0 & 0 & 0\end{bmatrix}\ . 
\edm
We will identify vector fields with $1$-forms via the Riemannian metric,
thus obtaining a  coframe $\eta_1, \, \eta_2 ,  \ldots , \eta_7$, and shall
use throughout the abbreviation
$\eta_{i j \ldots} := \eta_i \wedge \eta_j \wedge \ldots $. In this frame, we 
compute the differentials $d \eta_{\alpha}$,
\begin{eqnarray*}
d \eta_1 &=& - \, 2\, ( \eta_{23} + \eta_{45} + \eta_{67})  , \\ 
d \eta_2 &=& \ \  \, 2\, ( \eta_{13} - \eta_{46} + \eta_{57})  , \\
d \eta_3 &=& - \, 2\, ( \eta_{12} + \eta_{47} + \eta_{56})  . 
\end{eqnarray*}
Each of the three Sasaki structures on $M^7$ admits a characteristic 
connection, i.\,e.~a metric connection with antisymmetric torsion;
however, this torsion is well-known to be $\eta_i\wedge d\eta_i$ 
\cite[Thm 8.2]{FriedrichIvanov}, and these
do not coincide for $i=1,2,3$. Thus, a $3$-Sasakian manifold has no
characteristic connection \cite[ \S 2.6]{Srni}. 

\section{The canonical $\G_2$-structure of a $3$-Sasakian manifold}%
\noindent
%
Consider the following $3$-forms,
\bdm
F_1 \ := \ \eta_1 \wedge \eta_2 \wedge \eta_3 \, , \quad
F_2 \ := \ \frac{1}{2}\big( \eta_1 \wedge d \eta_1  +  \eta_2 \wedge 
d \eta_2 + \eta_3 \wedge d \eta_3 \big)  +  3  \eta_1 \wedge \eta_2
\wedge \eta_3.
\edm
Then
\bdm
\omega \ := \ F_1 + F_2 \ = \ \eta_{123} - \eta_{145} - \eta_{167} - 
\eta_{246} +
\eta_{257}  - \eta_{347} - \eta_{356} 
\edm
is a generic $3$-form defined globally on $M^7$. It
induces a  $\G_2$-structure on $M^7$.
\begin{dfn}
The $3$-form $\omega = F_1 + F_2$ is called the 
\emph{canonical $\G_2$-structure} of the $7$-dimensional $3$-Sasakian manifold.
\end{dfn}
\noindent
We investigate now  the type of this canonical $\G_2$-structure from the 
point of view of $\G_2$-geometry \cite{FerGray},
\cite{FriedrichIvanov}. It is basically described by the differential of the 
$\G_2$-structure $\omega$. We compute directly \cite{FKMS}
\bdm
d F_1 \ = \ 2 \cdot (* F_2)  , \quad d F_2 \ = \ 12 \cdot (* F_1)  +  2
\cdot (* F_2) \, , \quad d*F_1 \ = \ d * F_2 \ = \ 0  .
\edm
 In particular, the canonical $\G_2$-structure is cocalibrated. Equivalently,
 it is of type
 $\mathcal{W}_1 \oplus \mathcal{W}_3 = \Lambda^3_1 \oplus \Lambda^3_{27}$ in
 the Fernandez/Gray notation, see \cite{FerGray}, \cite{FriedrichIvanov}, \cite{FriedrichIvanov2},
\bdm
d * \omega \ = \ 0 , \quad * d \omega \ = \ 4 \, ( 3 \, F_1  +  F_2)  . 
\edm
There exists a unique connection $\nabla^c$ preserving the 
$\G_2$-structure with totally skew-symmetric torsion $\mathrm{T}^c$
\cite{FriedrichIvanov}, \cite{FriedrichIvanov2}. For a cocalibrated
$\G_2$-structure $\omega$ this 
{\it characteristic torsion form $\mathrm{T}^c$} is given by the formula
\bdm
\mathrm{T}^c \ = \ - \, * d \omega  +  \frac{1}{6} ( d \omega  , \ *
\omega) \cdot \omega .
\edm 
We express the characteristic torsion by the data of the $3$-Sasakian
structure,
\bdm
\mathrm{T}^c \ = \ -  6  F_1  +  2  F_2 \ = \ \eta_1 \wedge d \eta_1
 +  \eta_2 \wedge d \eta_2  +  \eta_3 \wedge d \eta_3 \ = \ 2 
\omega  -  8 \, F_1  .
\edm
Thus, we see that $\mathrm{T}^c$ is the sum of
the three characteristic torsion forms of the Sasakian structures $\eta_i$.\\

\noindent
Let us decompose the characteristic torsion $\mathrm{T}^c =
\mathrm{T}^c_1 + \mathrm{T}^c_{27}$ into the $\mathcal{W}_1= \Lambda^3_1$- 
and the
$\mathcal{W}_3 = \Lambda^3_{27}$ -part , respectively. Then we obtain
\bdm
\mathrm{T}^c_1 \ = \ \frac{6}{7} \, (F_1 \, + \, F_2) \ = \ \frac{6}{7} \, 
\omega, \quad
\mathrm{T}^c_{27}\ = \ \frac{8}{7} \, (F_2 \, - \, 6 \, F_1) .
\edm
In particular, the canonical $\G_2$-structure of a $3$-Sasakian manifold 
is never of pure type  $\mathcal{W}_1$ or $\mathcal{W}_3$.
\\

\noindent
We will now prove that the  canonical $\G_2$-structure has parallel characteristic torsion, $\nabla^c
\mathrm{T}^c = 0$, and realizes one type of cocalibrated $\G_2$-structures
with characteristic holonomy contained in the maximal, six-dimensional
subalgebra $\su(2) \oplus \su_c(2)$ of $\g_2$ \cite{Fri2}. Later,
we shall see that its holonomy algebra coincides with  $\su(2) \oplus \su_c(2)$.
\begin{thm}
The canonical $\G_2$-structure $\omega$ of a $7$-dimensional $3$-Sasakian
manifold is cocalibrated, $d * \omega = 0$. Its characteristic torsion is
given by the formula
\bdm
\mathrm{T}^c \ = \ - \, * d \omega + 6   \omega  .
\edm
Moreover, we have $(d \omega , * \omega) = 36$, $|\mathrm{T}^c|^2 = 60$
and 
\bdm
d * \mathrm{T}^c \ = \ 0  , \quad d \mathrm{T}^c \ = \ - \, 4 \, *
\mathrm{T}^c  , \quad d \omega \ = \ \frac{1}{2} \, d * d \omega \, - \, 12
\, * \omega  .
\edm
The characteristic connection preserves the splitting $TM^7 = \mathrm{T}^v
\oplus \mathrm{T}^h$ and the characteristic torsion is $\nabla^c$-parallel, 
$\nabla^c \mathrm{T}^c = 0$.
\end{thm}
\begin{proof}
Since $\xi_1$ is a Killing vector field, we have
\bdm
\nabla^g_X \eta_1 \ = \ \frac{1}{2} \, X \haken d \eta_1  .
\edm
Then we obtain
\begin{eqnarray*}
\nabla^c_X \eta_1 \ = \ \nabla^g_X \eta_1 \, + \, \frac{1}{2} \,
\mathrm{T}^c(X,\eta_1, -) \ = \ \frac{1}{2} \, X \haken d \eta_1 \, - \, 
\frac{1}{2} \, X \haken ( \eta_1 \haken \mathrm{T}^c ) \ .
\end{eqnarray*}
The formula $\mathrm{T}^c = \eta_1 \wedge d \eta_1 + \eta_2 \wedge d
\eta_2 + \eta_3 \wedge d \eta_3$ yields directly
\bdm
\eta_1 \haken \mathrm{T}^c \ = \ d \eta_1 + (\eta_1 \haken d \eta_2)
\wedge \eta_2 +  (\eta_1 \haken d \eta_3)
\wedge \eta_3 . 
\edm 
Moreover, the formulas for the differential $d \eta_{\alpha}$ imply that
\bdm
\eta_1 \haken d \eta_2 \ = \  2 \, \eta_3 , \quad
\eta_1 \haken d \eta_3 \ = \ - \, 2 \, \eta_2
\edm
holds. Thus we obtain
\bdm
\nabla^c_X \eta_1 \ = \ 2 \, X \haken ( \eta_2 \wedge \eta_3) ,
\edm
i.\,e.~$\nabla^c$ preserves the subbundle $\mathrm{T}^v$. Finally we have
\bdm
(\nabla^c_X \eta_1) \wedge \eta_2 \wedge \eta_3 \ = \ 0 
\edm
and then $\nabla^c(\eta_1 \wedge \eta_2 \wedge \eta_3) = 0$. Since 
$\mathrm{T}^c \, = \,  2
\omega - 8 \, \eta_1 \wedge \eta_2 \wedge \eta_3$ and $\nabla^c \omega = 0$ we
conclude that $\nabla^c \mathrm{T}^c = 0$ holds, too.
\end{proof}
%
\section{The canonical spinor of a $3$-Sasakian manifold}\noindent
%
Since the spin representation of $\Spin(7)$ is real, let us consider
the real spinor bundle $\Sigma$. Any $\G_2$-structure $\omega$
acts via the Clifford multiplication on $\Sigma$ as a symmetric
endomorphism with eigenvalue $(-7)$ of multiplicity one and eigenvalue $1$ of
multiplicity seven. Consequently, any $\G_2$-structure on a simply-connected 
manifold $M^7$ defines a \emph{canonical spinor field} $\Psi_0$ such that (see
\cite{FKMS}, \cite{FriedrichIvanov})
\bdm
\omega \cdot \Psi_0 \ = \ - \, 7 \, \Psi_0 \, , \quad |\Psi_0| \ = \ 1 \ .
\edm
If $(M^7 , \omega)$ is cocalibrated and $\nabla^c$ is its characteristic
connection , we obtain \cite{FriedrichIvanov}, \cite{AgFr2}
\bdm
\nabla^c \Psi_0 \ = \ 0 \, , \quad \mathrm{T}^c \cdot \Psi_0 \ = \ - \,
\frac{1}{6} \, (d \omega , * \omega) \cdot \Psi_0 \, , \quad
\mathrm{Scal}^g \ = \ \frac{1}{18} \, (d \omega , * \omega)^2 \, - \,
\frac{1}{2} \, |\mathrm{T}^c|^2 \ ,
\edm
We apply the general formulas
to the canonical spinor of a $3$-Sasakian manifold $M^7$. Then we obtain
a spinor field such that
\bdm
\omega \cdot \Psi_0 \ = \ - \, 7 \, \Psi_0 \, , \quad
\mathrm{T}^c \cdot \Psi_0 \ = \ - \, 6 \, \Psi_0 \, , \quad
\nabla^g_X \Psi_0 \, + \, \frac{1}{4} \, ( X \haken \mathrm{T}^c) \cdot 
\Psi_0 \ = \ 0 \ .
\edm
Using the explicit formulas for $\omega$ and $\mathrm{T}^c$, a direct
algebraic computation in the real spin representation yields the following
\begin{lem}
\begin{eqnarray*}
\mathrm{T}^c \cdot X \cdot \Psi_0 &=& -  \frac{5}{3}  X \cdot
\mathrm{T}^c \cdot \Psi_0 \ = \ 10 \, X \cdot \Psi_0 \quad \text{if} \quad X \in \mathrm{T}^v  , \\
\mathrm{T}^c \cdot X \cdot \Psi_0 &=&  X \cdot
\mathrm{T}^c \cdot \Psi_0 \ = \ - \, 6 \, X \cdot \Psi_0 \quad \quad \text{if} \quad X \in \mathrm{T}^h ,
\end{eqnarray*}
\end{lem}
\noindent
The equation $\nabla^c \Psi_0 = 0$ can be written as
\bdm
\nabla^g_X \Psi_0 \, - \, \frac{1}{8} \, (X \cdot \mathrm{T}^c \, + \,
\mathrm{T}^c \cdot X ) \cdot \Psi_0 \ = \ 0 \ .
\edm
We apply now the algebraic Lemma and obtain a differential equation involving
the canonical spinor of a $3$-Sasakian manifold.
\begin{thm}
The canonical spinor field $\Psi_0$ of a $7$-dimensional $3$-Sasakian manifold
satisfies the following differential equation:
\bdm
\nabla^g_X \Psi_0 \ = \ \frac{1}{2} \, X \cdot \Psi_0 \quad \mathrm{if} \quad
X \in \mathrm{T}^v \, , \quad \nabla^g_X \Psi_0 \ = \ - \, \frac{3}{2} \, X \cdot \Psi_0 \quad \mathrm{if} \quad
X \in \mathrm{T}^h \ .
\edm
In particular, $\Psi_0$ is an eigenspinor for the Riemannian Dirac operator,
$D^g \Psi_0 = \frac{9}{2} \, \Psi_0$.
\end{thm}
\begin{NB}
This equation has  already been discussed in \cite{Fri2}, section $10$. It
follows essentially from the formula $\mathrm{T}^c = 2 \, \omega - 8 \, F_1$.
\end{NB}
%
%
\section{$\nabla^c$-parallel vectors and spinors of the 
canonical $\G_2$-structure}\noindent
%
%
The spinor bundle splits into three subbundles, $\Sigma = \Sigma_1 \oplus 
\Sigma_3 \oplus \Sigma_4$, where
\bdm
\Sigma_1 \ := \ \R \cdot \Psi_0 , \quad 
\Sigma_3 \ := \ \big\{ X \cdot \Psi_0 \, :\, X \in \mathrm{T}^v \big\}, \quad
\Sigma_4 \ := \ \big\{ X \cdot \Psi_0 \, : \, X \in \mathrm{T}^h \big\} .
\edm
The characteristic connection preserves this splitting. Obviously, the
$3$-form $\omega$ acts as the identity on $\Sigma_3 \oplus \Sigma_4$,
while the torsion form satisfies 
\begin{lem}
The torsion form $\mathrm{T}^c$ acts on $\Sigma_3$ as a multiplication by 
$10$ and it acts on $\Sigma_1 \oplus \Sigma_4$ as a multiplication by $(-6)$.
\end{lem}
\noindent
Given the definition of $\Sigma_4$, it is now a crucial observation that
$\nabla^c$-parallel vector fields cannot be horizontal:
\begin{prop}\label{par-hor}
Horizontal, $\nabla^c$-parallel vector fields
\bdm
\nabla^c X \ = \ 0  , \quad 0 \ \neq \ X \in \Gamma(\mathrm{T}^c)
\edm
do not exist.
\end{prop}
\begin{proof}
Let $0 \neq X$ be the vector field. Then $\Psi : = X \cdot \Psi_0$ is a 
$\nabla^c$-parallel spinor, too. Moreover, the torsion form acts on $\Psi_0$
and on $\Psi$ by the same eigenvalue,
\bdm
\mathrm{T}^c \cdot \Psi_0 \ = \ - \, 6 \, \Psi_0 \, , \quad
\mathrm{T}^c \cdot \Psi \ = \ - \, 6 \, \Psi \ .
\edm 
The holonomy algebra $\hol(\nabla^c)$ is contained in $\su(2) \oplus \su_c(2)
\subset \g_2 \subset \so(7)$ and the linear holonomy representation splits
into $\R^7 = \R^4 \oplus \R^3$. The vector field $X$ is an element of $\R^4$
such that $\hol(\nabla^c) \cdot X = 0$. In \cite{Fri2} we explicitely realized
the Lie algebra   $\su(2) \oplus \su_c(2)$ inside  $\so(7)$. Using these
formulas, an easy computation yields  that the holonomy algebra is
contained in $\so(3) \subset \su(3) \subset \g_2$ and the linear holonomy
representation splits into $\R^7 = \R^3 \oplus \R^3 \oplus
\R^1$. Consequently, the $\G_2$-manifold $(M^7, \omega)$ is cocalibrated, its
characteristic holonomy is contained in $\so(3)$ and the characteristic
torsion $\mathrm{T}^c$ acts on both $\nabla^c$-parallel spinors with the same
eigenvalue. It turns out that $M^7$ cannot be an Einstein manifold with
positive scalar curvature by \cite[Thm 7.1]{Fri2}, a contradiction.
\end{proof}
\noindent
In general, the Casimir operator of a metric connection with parallel
characteristic torsion is given by the following formulas \cite{AgFr2}
\bdm
\Omega \ = \ (D^{1/3})^2  -  \frac{1}{16}  ( 2 \, \mathrm{Scal}^g  +
\, |\mathrm{T}^c|^2 ) \ = \ \Delta_{\mathrm{T}^c}  +   
\frac{1}{16} \, ( 2 \, \mathrm{Scal}^g  +  |\mathrm{T}^c|^2 )  - 
\frac{1}{4} \, (\mathrm{T}^c)^2  .
\edm
Its kernel contains the space of all $\nabla^c$-parallel spinor fields. In
particular, any $\nabla^c$-parallel spinor field $\Psi$ satisfies the algebraic
condition \cite{FriedrichIvanov}, \cite{AgFr2}
\bdm
4 \, (\mathrm{T}^c)^2 \cdot \Psi \ = \ ( 2 \, \mathrm{Scal}^g \, + \,
|\mathrm{T}^c|^2) \cdot \Psi .
\edm
For the canonical $\G_2$-structure of a $3$-Sasakian manifold we have
$ 2 \, \mathrm{Scal}^g  + |\mathrm{T}^c|^2 = 144$. Consequently, any 
$\nabla^c$-parallel spinor field is a section in the subbundle $\Sigma_1 \oplus
\Sigma_4$, i.\,e. of the form
$\Psi = a \cdot \Psi_0 + X \cdot \Psi_0$, where $a$ is constant and 
$X \in \Gamma(\mathrm{T}^h)$ is  a horizontal, parallel vector field. 
But horizontal, $\nabla^c$-parallel vector fields do not
exist. This argument proves:
\begin{thm}
Any $\nabla^c$-parallel spinor field is proportional to $\Psi_0$.
Moreover, the holonomy algebra is the six-dimensional maximal subalgebra 
  $\hol(\nabla^c) = \su(2) \oplus \su_c(2)$ of $\g_2$.
\end{thm}
\noindent
The latter argument proves that vertical, $\nabla^c$-parallel vector fields
do not exist. Indeed, if $\nabla^c X = 0$, then $X \cdot \Psi_0 \in
\Gamma(\Sigma_3)$ is a parallel spinor in $\Sigma_3$. We conclude that 
$X \cdot \Psi_0
= 0$ and  $X = 0$. Together with Proposition \ref{par-hor} and the splitting
of the tangent bundle, one concludes:
\begin{thm}
There are no non-trivial 
$\nabla^c$-parallel vector fields.
\end{thm}
%
\section{Riemannian Killing spinors on $3$-Sasakian manifolds}%
\noindent
%
Consider the spinor fields $\Psi_1 := \xi_1 \cdot \Psi_0  , \, 
\Psi_2 := \xi_2 \cdot \Psi_0  , \, \Psi_3 := \xi_3 \cdot \Psi_0$. These
spinors are sections in the bundle $\Sigma_3$.
\begin{thm}
The spinor fields $\Psi_{\alpha}$ are Riemannian Killing spinors, i.\,e.
\bdm
\nabla^g_X \Psi_{\alpha} \ = \ \frac{1}{2} \, X \cdot \Psi_{\alpha} \, , \quad
\alpha = 1,2, 3.
\edm
\end{thm}
\begin{cor}[\cite{FrKa1}, \cite{FrKa2}]
 Any simply-connected $3$-Sasakian
  manifold admits at least three Riemannian Killing spinors.
\end{cor}
\begin{proof}
We use the differential equation
\bdm
\nabla^g_X \Psi_0 \ = \ \frac{1}{8} \, (X \cdot \mathrm{T}^c + 
\mathrm{T}^c \cdot X ) \cdot \Psi_0 
\edm
as well as the properties of  Sasakian structures. Then we obtain
\begin{eqnarray*}
\nabla^g_X ( \xi_1 \cdot \Psi_0) &=& \big(\nabla^g_X \xi_1\big)\cdot \Psi_0 
+  \xi_1 \cdot \nabla^g_X \Psi_0 \\
&=& - 
\varphi_1(X) \cdot \Psi_0 + \frac{1}{8} \, \xi_1 \cdot (X \cdot \mathrm{T}^c +
\mathrm{T}^c \cdot X ) \cdot \Psi_0  \\
&=&  \ \frac{1}{2} \, \big( X \haken d \eta_1 \big) \cdot \Psi_0  +  
\frac{1}{8} \, \xi_1 \cdot (X \cdot \mathrm{T}^c + 
\mathrm{T}^c \cdot X ) \cdot \Psi_0  \\
&=& -  \frac{1}{4} \big( X \cdot d \eta_1  -  d \eta_1 \cdot X \big)
\cdot \Psi_0 + \frac{1}{8} \, \xi_1 \cdot (X \cdot \mathrm{T}^c  + 
\mathrm{T}^c \cdot X ) \cdot \Psi_0 .
\end{eqnarray*}
A direct algebraic computation yields now that 
\bdm
-  \frac{1}{4} \big( X \cdot d \eta_1  - d \eta_1 \cdot X \big)
\cdot \Psi_0  +  \frac{1}{8} \, \xi_1 \cdot (X \cdot \mathrm{T}^c +
\mathrm{T}^c \cdot X ) \cdot \Psi_0  \ = \ \frac{1}{2} \, X \cdot \xi_1 \cdot
\Psi_0 
\edm
holds specially for the spinor $\Psi_0$. This proves the statement of the Theorem.
\end{proof}
\noindent
In general, any real spinor field $\Phi$ of length one defined 
on a $7$-dimensional
Riemannian manifold induces a $\G_2$-structure $\omega_{\Phi}$ 
(see \cite{FKMS}). Moreover, if
two spinor fields $\Phi_2 = \xi \cdot \Phi_1$ are related via  Clifford
multiplication by some vector field $\xi$, then 
\bdm
\omega_{\Phi_2} \ = \ - \, \omega_{\Phi_1}  + 2  ( \xi \haken
\omega_{\Phi_1}) \wedge \xi 
\edm
holds  \cite[Remark 2.3]{FKMS}. Denote by $\omega_{\alpha}$ the
nearly parallel $\G_2$-structure induced by the Riemannian Killing spinor
$\Psi_{\alpha} = \xi_{\alpha} \cdot \Psi_0$ ($\alpha =1,2,3$). Then we
obtain
\bdm
\omega_{\alpha} \ = \ -  \frac{1}{2} \, ( \eta_1 \wedge d \eta_1 + \eta_2
\wedge d \eta_2 + \eta_3 \wedge d \eta_3) - 4 \, \eta_1 \wedge \eta_2 \wedge 
\eta_3 + 2 ( \xi_{\alpha} \haken \omega ) \wedge \eta_{\alpha}  .
\edm
Consider, for example, the case $\alpha = 1$. Then
\bdm
\xi_1 \haken \omega \ = \ \frac{1}{2} \, d \eta_1  +  \frac{1}{2} 
(\xi_1 \haken d \eta_2) \wedge \eta_2  +  \frac{1}{2} \, (\xi_1 \haken 
d \eta_3) \wedge \eta_3  +  4 \, \eta_{23} \ = \ \frac{1}{2} \, d \eta_1 +
 2 \, \eta_{23} .
\edm
Inserting the latter formula, we obtain
\begin{eqnarray*}
\omega_1 &=& \frac{1}{2} \, \eta_1 \wedge d \eta_1  -  \frac{1}{2} \,
\eta_2 \wedge d \eta_2  -  \frac{1}{2} \, \eta_3 \wedge d \eta_3 \\
&=& \eta_{123}  -  \eta_{145}  -  \eta_{167}  +  \eta_{246} - \eta_{257}  
+  \eta_{347} +  \eta_{356} .
\end{eqnarray*}
\begin{thm}
The nearly parallel $\G_2$-structures $\omega_1, \omega_2, \omega_3$ induced by the Killing spinors 
of a $3$-Sasakian manifold are given by the formulas
\begin{eqnarray*}
\omega_1 &=& \frac{1}{2} \, \eta_1 \wedge d \eta_1 \, - \, \frac{1}{2} \,
\eta_2 \wedge d \eta_2 \, - \, \frac{1}{2} \, \eta_3 \wedge d \eta_3 \\
\omega_2 &=& - \, \frac{1}{2} \, \eta_1 \wedge d \eta_1 \, + \, \frac{1}{2} \,
\eta_2 \wedge d \eta_2 \, - \, \frac{1}{2} \, \eta_3 \wedge d \eta_3 \\
\omega_3 &=& - \, \frac{1}{2} \, \eta_1 \wedge d \eta_1 \, - \, \frac{1}{2} \,
\eta_2 \wedge d \eta_2 \, + \, \frac{1}{2} \, \eta_3 \wedge d \eta_3 \ .
\end{eqnarray*}
All three nearly parallel $\G_2$-structures satisfy the equation $d
\omega_{\alpha} = - \, 4 \, ( * \omega_{\alpha})$.
\end{thm}
\begin{NB}
The nearly parallel structures $\omega_{\alpha}$ admit characteristic
connections, too. Their characteristic torsions $\mathrm{T}^c_{\alpha}$ are
proportional to $\omega_{\alpha}$ \cite{FriedrichIvanov}. Moreover, the
existence of a nearly parallel $\G_2$-structure or---equivalently---of a
Riemannian Killing spinor implies that $M^7$ is Einstein
\cite{Fri1}. Consequently, our construction explains why $3$-Sasakian manifolds
are Einstein manifolds.
\end{NB}
\section{Deformations of the canonical $\G_2$-structure}\noindent
%
Deformations of $3$-Sasakian metrics from the viewpoint of $\G_2$-geometry
have been studied in \cite{FKMS} and \cite{Fri2}. We once again describe the
construction of these particular $\G_2$-structures and their properties
in a unified way, and add some more. Fix a positive parameter $s > 0$ and
consider a new Riemannian metric $g^s$ defined by
\bdm
g^s(X  ,  Y) \ := \ g(X  ,  Y) \quad \mbox{if} \quad X,Y \in 
\mathrm{T}^h  , \quad 
g^s(X , Y) \ := \ s^2 \cdot g(X , Y) \quad \mbox{if} \quad X,Y \in 
\mathrm{T}^v .
\edm
Then $s \eta_1 , s \eta_2  , s  \eta_3  ,
\eta_4 , \ldots , \eta_7$ is an orthonormal coframe and we replace the $3$-forms
\begin{eqnarray*}
F_1 &=& \eta_1 \wedge \eta_2 \wedge \eta_3 \  , \\
F_2 &=& \frac{1}{2}\big( \eta_1 \wedge d \eta_1 + \eta_2 \wedge 
d \eta_2  + \eta_3 \wedge d \eta_3 \big)  + 3  \eta_1 \wedge \eta_2
\wedge \eta_3 \\
&=&  - \eta_{145} - \eta_{167} - \eta_{246} + \eta_{257}  - \eta_{347} - 
\eta_{356} 
\end{eqnarray*}
by the new forms
\bdm
F_1^s \ := \ s^3 F_1  , \quad F_2^s \ := \ s F_2 , \quad \omega^s \ := \
F_1^s  +  F_2^s .
\edm
$(M^7, g^s, \omega^s)$ is a Riemannian $7$-manifold equipped with a
$\G_2$-structure $\omega^s$. Denote by $*_s$ the corresponding Hodge operator
acting on forms. We summarize some well-known properties of these
$\G_2$-structures that follow from a straightforward computation.
\begin{thm}[\cite{FKMS}, Theorem $5.4$ and \cite{Fri2}, \S $10$]
\begin{enumerate}
\item[]
\item The $\G_2$-manifold $(M^7,g^s,\omega^s)$ is cocalibrated, $d *_s
  \omega^s = 0$.
\item The differential of the $\G_2$-structure is given by the formula
\bdm
d \omega^s \ = \ 12 \, s \, (*_s F_1^s) + \big(2s + \frac{2}{s} \big)
( *_s F_2^s) .
\edm
\item
The characteristic torsion $\mathrm{T}^c_s$ is given by the formula
\bdm
\mathrm{T}^c_s \ = \ \big( \frac{2}{s} - 10 s) (s\eta_1) \wedge (s \eta_2)
\wedge (s \eta_3) + 2 s \omega^s .
\edm
\item The Riemannian Ricci tensor is given by the formula
\bdm
\mathrm{Ric}^{g^s} \ = \ 6 \, (2 \, - \, s^2 ) \, 
\mathrm{Id}_{\mathrm{T}^h}  \, \oplus \, 
  \frac{2  +  4 s^4}{s^2} \, \mathrm{Id}_{\mathrm{T}^v}   .
\edm
In particular, the scalar curvature equals
\bdm
\mathrm{Scal}^{g^s} \ = \ 6 \, (8  +  \frac{1}{s^2}  -  2 s^2) .
\edm
\item
The canonical spinor field $\Psi_0$ satisfies the differential equation
\begin{eqnarray*}
\nabla^{g^s}_X \Psi_0 &=& -  \frac{3}{2} s \, X \cdot \Psi_0 \quad \mbox{if}
\quad X \in \mathrm{T}^h  , \\ 
\nabla^{g^s}_X \Psi_0 &=& \big( -  \frac{1}{2s}  +  s \big) \, X \cdot
\Psi_0 \quad \mbox{if} \quad X \in \mathrm{T}^v  .
\end{eqnarray*}
\end{enumerate} 
\end{thm}
\begin{cor}[\cite{FKMS}, Theorem $5.4$]
For $s = 1/\sqrt{5}$ the $\G_2$-structure is nearly parallel and $\Psi_0$ is a
Riemannian Killing spinor,
\bdm
d \omega^s \ = \ \frac{12}{\sqrt{5}} \, (*_s \omega^s)  , \quad
  \mathrm{Ric}^{g^s} \ = \ \frac{54}{5} \, \mathrm{Id}  , \quad 
\nabla^{g^s}_X \Psi_0 \ = \ -  \frac{3}{2 \sqrt{5}} \, X \cdot \Psi_0  .
\edm
$\Psi_0$ is the unique Riemannian Killing spinor of the metric.
\end{cor}
\begin{NB}
The Ricci tensor of the characteristic connection of $(M^7,
g^s, \omega^s)$ is given by the formula  \cite{Fri2} 
\bdm
\mathrm{Ric}^{\nabla^{c,s}} \ = \ 12 \, (1 \, - \, s^2 ) \, 
\mathrm{Id}_{\mathrm{T}^h}  \, \oplus \, 
  16 \, (1 \, - \, 2 \, s^2) \, \mathrm{Id}_{\mathrm{T}^v}   .
\edm
If $ s = 1$ (the $3$-Sasakian case), then $\mathrm{Ric}^{\nabla^{c}}$ vanishes
on the subbundle $\mathrm{T}^h$. For $s = 1/\sqrt{5}$, the Ricci tensor is
proportional to the metric, $\mathrm{Ric}^{\nabla^{c,1/\sqrt{5}}}= (48/5) \,
\mathrm{Id}_{TM^7}$. From this point of view there is a third interesting
parameter, namly $s^2 = 1/2$. Then the $\nabla^c$-Ricci tensor
vanishes on the subbundle  $\mathrm{T}^v$ and the canonical spinor field
$\Psi_0$ is parallel in vertical directions. It is a transversal Killing
spinor with respect to the three-dimensional foliation and
\bdm
\big(D^{g^s}\big)^2 \Psi_0 \ = \ 18 \, \Psi_0 \ = \ \frac{1}{4} \frac{4}{4-1}
\mathrm{Scal}^{g^s}\Psi_0 .
\edm
In particular, $\Psi_0$ is the first known example to realize the lower 
bound for the basic Dirac operator
of the foliation, see the recent work by Habib and Richardson \cite{HabibR}.
\end{NB}
    

\begin{thebibliography}{11}
\bibitem{Srni}
I. Agricola, \emph{The Srn\'{\i} lectures on non-integrable geometries with
torsion},  Arch. Math. (Brno) 42 (2006), 5-84.
%
\bibitem{AgFr1}
I. Agricola and Th. Friedrich, \emph{On the holonomy of connections
with skew-symmetric torsion}, Math. Ann. 328 (2004), 711-748.
\bibitem{AgFr2}
I. Agricola and Th. Friedrich, \emph{The Casimir operator of a metric 
connection with skew-symmetric torsion}, J. Geom. Phys. 50 (2004), 188-204.
\bibitem{BG}
C. Boyer and K. Galicki, \emph{Sasakian Geometry}, Oxford Mathematical 
Monographs, Oxford Univ. Press, 2008. 
\bibitem{FerGray}
M. Fernandez and A. Gray, \emph{Riemannian manifolds with structure
group $\G_2$}, Annali di Math. Pura e Appl. 132 (1982), 19-45.
\bibitem{Fri1}
Th. Friedrich, \emph{Der erste Eigenwert des Dirac-Operators einer kompakten
Riemannschen Mannigfaltigkeit nichtnegativer Skalarkr\"ummung}, Math. Nachr. 
97 (1980), 117-146.
\bibitem{Fri2}
Th. Friedrich, \emph{$\G_2$-manifolds with parallel characteristic torsion}, 
J. Diff. Geom. Appl. 25 (2007), 632-648.
\bibitem{FriedrichIvanov}
Th. Friedrich and S. Ivanov, \emph{Parallel spinors and connections with
skew-symmetric torsion in string theory}, Asian J. Math. 6 (2002), 
303-336.
\bibitem{FriedrichIvanov2}
Th. Friedrich ans S. Ivanov, \emph{Killing spinor equation in dimension $7$
  and geometry of integrable $\G_2$-manifolds}, J. Geom. Phys. 48 (2003),
1-11.
\bibitem{FrKa1}
Th. Friedrich and I. Kath, \emph{Varietes riemanniennes compactes de dimension
  $7$  admettant des spineurs de Killing}, C.R. Acad. Sci Paris 307 Serie I
(1988), 967-969.
\bibitem{FrKa2}
Th. Friedrich and I. Kath, \emph{Compact seven-dimensional manifolds with
  Killing spinors}, Comm. Math. Phys. 133 (1990), 543-561.
\bibitem{FKMS}
Th. Friedrich, I. Kath, A. Moroianu, and 
U. Semmelmann, \emph{On nearly parallel 
$\G_2$-structures}, J. Geom. Phys. 23 (1997), 256-286.
\bibitem{HabibR}
G. Habib and K. Richardson, \emph{A brief note on the spectrum of the basic
  Dirac operator}, preprint arXiv:0809.2406 (2008).
\bibitem{IK}
S. Ishihara and M. Konish, \emph{Differential geometry of fibred spaces},
Kyoto 1973.
\end{thebibliography}
\end{document}